\documentclass[12pt, reqno]{amsart}
\usepackage{ amsthm,latexsym, amscd, amsfonts, amssymb, graphicx, color}
\usepackage[bookmarksnumbered, colorlinks, plainpages, backref]{hyperref}

\definecolor{TITLE}{rgb}{0.0,0.0,1.0}
\definecolor{AUTHOR1}{rgb}{0.00,0.59,0.00}
\definecolor{AUTHOR2}{rgb}{0.50,0.00,1.00}
\definecolor{SECTION}{rgb}{0.50,0.00,1.00}
\definecolor{FOOTTITLE}{rgb}{0.00,0.50,0.75}
\definecolor{THM}{rgb}{0.7,0.3,0.3}
\definecolor{SEC}{rgb}{0.6,0.1,.5}

\makeatletter
\@namedef{subjclassname@2010}{%
  \textup{2010} Mathematics Subject Classification}
\makeatother
\newtheorem{theorem}{{\color{THM} Theorem}}[section]

\let\oldthm\thetheorem
\def\thetheorem{{\color{THM}\oldthm}}

\newtheorem{lemma}[theorem]{{\color{THM}Lemma}}

\newtheorem{corollary}[theorem]{{\color{THM}Corollary}}
\theoremstyle{definition}
\newtheorem{definition}[theorem]{{\color{THM}Definition\ }}
\newtheorem{example}[theorem]{{\color{THM}Example}}

\newtheorem{remark}[theorem]{{\color{THM}Remark}}
\numberwithin{equation}{section}

\newcommand{\oj}{\overrightarrow{{\omega}}}
\newcommand{\Oj}{\overrightarrow{{\Omega}}}

\newcommand{\A}{\mathfrak A}

\newcommand{\bea}{\begin{eqnarray*}}
\newcommand{\eea}{\end{eqnarray*}}
\numberwithin{equation}{section}
\begin{document}
\title[The Arens  Regularity ...]{{\color{TITLE} Arens regularity of certain weighted semigroup algebra and countability}}
\author[ Khodsiani , Rejali and Ebrahimi Vishki]{\color{AUTHOR2} B. Khodsiani$^1$,  A. Rejali$^2$  and H.R. Ebrahimi Vishki$^3$}
\address{$^{1,2}$ Department of Mathematics, University of Isfahan, Isfahan, IRAN.}
\email{b-khodsiani@sci.ui.ac.ir}\email{rejali@sci.ui.ac.ir}
\address{$^3$ Department of Pure Mathematics and  Centre of Excellence
in Analysis on Algebraic Structures (CEAAS), Ferdowsi University of
Mashhad, P.O. Box 1159, Mashhad 91775, Iran.}
\email{vishki@um.ac.ir}



\begin{abstract} It is known that every countable semigroup admits a weight $\omega$ for which the semigroup algebra $\ell_1(S,\omega)$ is Arens regular and no uncountable group admits such a weight; see \cite{CY}.  In this paper, among other things,  we show that for a large class of semigroups, the Arens regularity of the weighted semigroup algebra $\ell_1(S,\omega)$ implies the countability of $S$.
 \end{abstract}
\maketitle

\section{\color{SEC}Introduction and Preliminaries}
Arens \cite{ARENS2}  introduced two multiplications on the second dual
$\A^{**}$ of a Banach algebra $\A$ turning it into Banach algebra. If these multiplications are coincide then $\A$ is said to be Arens regular. The Arens regularity of the semigroup algebra $\ell_1(S)$ has been investigated in \cite{Young}.   The Arens regularity of the weighted semigroup algebra $\ell_1(S,\omega)$ has been studied in \cite{CY} and \cite{BR}. In \cite{BR} Baker  and Rejali   obtained some nice  criterions for Arens regularity of
$\ell_1(S,\omega)$. Recent developments on the Arens regularity of $\ell_1(S,\omega)$ can be found in \cite{DL}.
For the algebraic theory of semigroups our general reference is \cite{Howie}.

In this paper we first show that the Arens regularity of a weighted semigroup algebra is stable under certain homomorphisms of semigroups (Lemma \ref{epimorphism}). Then we study those conditions under which the Arens regularity of $\ell_1(S,\omega)$ necessities the countability of $S$. The most famous example for such a semigroup is actually a group, as Craw and Young have proved in their nice paper \cite{CY}. As the main aim of the paper we shall show that for a wide variety of semigroups the Arens regularity of $\ell_1(S,\omega)$ implies that $S$ is countable; (see Theorems \ref{main} and \ref{main1}).

\section{ Arens Regularity of $\ell_1(S,\omega)$ and some hereditary properties}
Let $S$ be a semigroup  and $\omega: S\rightarrow(0,\infty)$ be a weight on $S$, i.e. $\omega(st)\leq\omega(s)\omega(t)$ for all
$s,t\in S$, and let $\Omega: S\times S\rightarrow(0,1]$ be defined by $\Omega(s,t)=\frac{\omega(st)}{\omega(s)\omega(t)}$, for $s,t\in
S$. Following \cite{BR}, we call $\Omega$ to be  $0-$cluster if for each pair of sequences $(x_n),(y_m)$ of
distinct elements of  $S$, $\lim_n\lim_m\Omega(x_n,y_m)=0=\lim_m\lim_n\Omega(x_n,y_m)$ whenever both iterated limits exist.

We define, $$\ell_\infty(S,\omega):=\{f:S\rightarrow \mathbb{C}: ||f||_{\omega,\infty}=\sup\{|\frac{f(s)}{\omega(s)}|:s\in
S\}<\infty\}$$
$$\ell_1(S,\omega):=\{g:S\rightarrow \mathbb{C}: ||g||_{\omega,1}=\sum_{s\in S}|g(s)|\omega(s)<\infty\}.$$

For ease of reference we quote the following criterion from \cite{BR} which will be frequently used in the sequel.
 \begin{theorem}\label{BR}\cite[Theorems 3.2, 3.3]{BR}
 For a weighted semigroup algebra  $\ell_1(S,\omega),$ the following statements are equivalent.
 \begin{enumerate}
 \item [(i)] $\ell_1(S,\omega)$ is  regular.
 \item [(ii)] The map $(x,y)\mapsto\chi_A(xy)\Omega(x,y)$ is cluster on $S\times S$ for each $A\subseteq S.$
 \item [(iii)] For each pair of sequences $(x_n),(y_m)$ of distinct points of $S$ there exist subsequences $(x'_n),(y'_m)$ of $(x_n),(y_m)$ respectively such that either\begin{enumerate}
     \item[(a)] $\lim_n\lim_m\Omega(x'_n,y'_m)=0=\lim_m\lim_n\Omega(x'_n,y'_m),$ or
      \item[(b)] the matrix $(x'_ny'_m)$ is of type $C$.
      \end{enumerate}
\end{enumerate}
In particular, if $\Omega$ is $0-$cluster then $\ell_1(S,\omega)$ is  regular.
 \end{theorem}

 Let $\psi: S\rightarrow T$ be a homomorphism  of semigroups. If   $\omega$ is a  weight on $T$ then trivially $\overleftarrow{{\omega}}(s):=\omega(\psi(s))$ defines a weight on $S$.

 If  $\psi: S\rightarrow T$ is an epimorphism and $\omega$ is a   bounded below (that is, $\inf\omega(S)>0$)  weight on $S$ then  a direct verification reveals that
\[\overrightarrow{{\omega}}(t):=\inf\omega(\psi^{-1}(t)),\quad (t\in T),\] defines a weight on $T.$
We commence with the next elementary result concerning to the stability of regularity under the semigroup homomorphism.
\begin{lemma}\label{epimorphism} Let $\psi: S\rightarrow T$ be a homomorphism of semigroups.
\begin{enumerate}
  \item [(i)] If $\psi$ is onto and  $\omega$ is bounded below weight on $S$ then the regularity of  $\ell_1(S,\omega)$ necessities the regularity of  $\ell_1(T,\overrightarrow{{\omega}})$.
  Furthermore if $\Omega$ is 0-cluster, then $\overrightarrow{{\Omega}}$ is 0-cluster.
  \item [(ii)] For a weight $\omega$ on $T$ if   $\ell_1(S,\overleftarrow{{\omega}})$
      is  regular, then $\ell_1(T,\omega)$ is  regular.
 \end{enumerate}
\end{lemma}

\begin{proof}(i) Since $\omega$ is bounded below,  we can assume that, $\inf\omega(S)\geq\varepsilon>0$, for some
$\varepsilon<1$. Hence $\oj\geq \varepsilon$. Let $(x_n),(y_m)$ be sequences of distinct elements in $T$. Then there are sequences
of distinct elements $(s_n), (t_m)$ in $S$ such that
 $$\left\{\begin{array}{lr}
                                                     \oj(x_n) >\omega(s_n)(1-\varepsilon) & \mbox{and}\quad \psi (s_n)=x_n, \\
                                                     \oj(y_m) >\omega(t_m)(1-\varepsilon) & \mbox{and}\quad \psi (t_m)=y_m.
                                                     \end{array}\right.
$$


It follows that $\oj(x_n) \oj(y_m)>\omega(s_n)\omega(t_m)(1-\varepsilon)^2$ and so from $\oj(x_ny_m)\leq \omega(s_nt_m)$ we get  $\frac{\oj(x_ny_m)}{\oj(x_n) \oj(y_m)}\leq\frac{1}{(1-\varepsilon)^2}\frac{ \omega(s_nt_m)}{\omega(s_n)\omega(t_m)};$
or equivalently,
 \begin{equation} \label{nonequa}
\Oj(x_n,y_m)\leq \frac{1}{(1-\varepsilon)^2}\Omega(s_n,t_m),\quad (n,m\in
\mathbb{N}).
\end{equation}
Applying the inequality \eqref{nonequa}, an standard argument based on Theorem \ref{BR} shows that if   $\ell_1(S,\omega)$
      is  regular then  $\ell_1(T,\overrightarrow{{\omega}})$ is  regular.
\end{proof}
\begin{corollary}\label{weight}
Let $\psi :S\rightarrow T$ be a homomorphism of semigroups. If $\ell_1(S)$ is  Arens regular then  $\ell_1(T,\omega)$ is Arens regular, for every  weight function $\omega$ on $T$.
\end{corollary}
\begin{proof}
Let $\ell^1(S)$ be Arens regular and let $\omega$ be a weight on $T$. Then  $\ell^1(S,\overleftarrow{\omega})$ is Arens regular by  \cite[Corollary 3.4]{BR}. Lemma \ref{epimorphism} implies that $\ell_1(T,\omega)$ is Arens regular.
\end{proof}

\section{Arens regularity of $\ell_1(S,\omega)$ and countability of $S$}
We commence with  the  next result of Craw and Young with a slightly simpler proof.
\begin{corollary}\label{countable}(See \cite[Corollary 1]{CY})
 Let $S$ be a countable semigroup. Then there exists a bounded below weight  $\omega$ on $S$ such that $\Omega$ is 0-cluster. In particular, $\ell_1(S,\omega)$ is Arens regular.
\end{corollary}
\begin{proof}  Let $F$ be the free semigroup generated by the countable semigroup $S=\{a_k:\quad k\in\mathbb{N}\}$.  For every element $x\in F$ (with the  unique presentation  $x=a_{k_1}a_{k_2}\cdots a_{k_r}$) set $\omega_1(x)=1+k_1+k_2+\cdots k_r$. A direct verification shows that $\omega_1$ is a weight on $F$  with $1\leq\omega_1,$ and that $\Omega_1$ is $0-$cluster.  Let $\psi: F\rightarrow S$ be the canonical epimorphism.  Set $\omega:=\overrightarrow{{\omega_1}}.$ By Lemma \ref{epimorphism}, $\omega$ is our desired weight on $S$.
\end{proof}

In the sequel the following elementary lemma will be frequently used.
\begin{lemma}\label{countable1}
A nonempty set $X$ is countable if and only if there exists a  function $f:X\rightarrow(0,\infty)$ such that the
sequence $(f(x_n))$ is unbounded for every  sequence $(x_n)$ with distinct elements in $X.$
\end{lemma}
\begin{proof}  If   $X=\{x_n:n\in \mathbb{N}\}$ is countable the  $f(x_n)=n$ is the desired function. For the converse, suppose that such a function $f:X\rightarrow(0,\infty)$ exists. Since $X=\cup_{n\in\mathbb N}\{x\in X:f(x)\leq n \}$ and each of the sets $\{x\in X:f(x)\leq n\}$ is countable,   so $X$ is countable.
\end{proof}
\begin{theorem}\label{cluster}
If $\ell^1(S)$ is not Arens regular and $S$ admits a bounded below weight for which $\Omega$ is $0-$cluster,  then $S$ is countable.
\end{theorem}
\begin{proof}
Let $\omega$ be a  bounded below weight for which $\Omega$ is $0-$cluster. Let $\epsilon>0$ is so that $\omega\geq\epsilon.$ Let $S$ be uncountable. By  Lemma \ref{countable1} there is a sequence $(s_n)$ of distinct
elements in $S$ and $n_0\in \mathbb{N}$ such that $\omega(s_n)\leq n_0$ for all $n\in \mathbb{N}$. As  $\ell_1(S)$ is not Arens regular,
there exist subsequences $(s_{n_k}),(s_{m_l})$ of $(s_n)$ such that $\{s_{n_k}s_{m_l}:k<l\}\cap\{s_{n_k}s_{m_l}:k>l\}=\emptyset$ (????????????????). We thus get
\[\Omega(s_{n_k},s_{m_l})=\frac{\omega(s_{n_k}s_{m_l})}{\omega(s_{n_k})\omega(s_{m_l})}\geq\frac{\epsilon}{n_0^2},\quad (k,l\in\mathbb{N}),\]
 contradicts the $0-$clusterlity of   $\Omega$.
\end{proof}
Abtehi et al. \cite{A.Kh.R} have shown that for a wide variety of semigroups (including Brandt semigroups, weakly cancellative semigroups, (0-)simple inverse semigroups and inverse semigroups with finite set of idempotents) the Arens regularity of the semigroup algebra $\ell^1(S)$ necessities the finiteness of $S$ (see \cite[Corollary 3.2, Proposition 3.4 and Theorem 3.6]{A.Kh.R}). Applying these together with Theorem \ref{cluster} we arrive to the next result.

Note that as it has been reminded  in Theorem \ref{BR}, if $\Omega$ is $0-$cluster then $\ell_1(S,\omega)$ is  regular and the converse is also true in the case where $S$ is weakly cancellative; (see \cite[Corollary 3.8]{BR}).

\begin{theorem}\label{main}
If $S$ admits a bounded below weight for which $\Omega$ is $0-$cluster then $S$ is countable in either of the following cases.
\begin{enumerate}

 \item $S$ is a Brandt semigroup.
  \item $S$ is weakly cancellative.
  \item $S$ is a simple (resp. $0-$simple) inverse semigroup.
  \item $S$ is an inverse semigroup with finitely many idempotents.
  \end{enumerate}
\end{theorem}
In the next result we shall show that the same result holds when $S$ is a completely simple semigroup.
\begin{theorem}\label{main1}
If $S$ admits a bounded below weight for which $\Omega$ is $0-$cluster then $S$ is countable in the case where $S$ is completely simple [resp. 0-simple].
\end{theorem}

\begin{proof}
Suppose that $\omega$ is a bounded below weight on $S$ such that $\Omega$ is $0-$cluster.  Let   $S$ be completely $0-$simple, then as it has been explained in \cite{Howie}, $S$ has the presentation  $S\cong M^0(G,I,\Lambda;P)=(I\times G\times \Lambda)\cup \{0\}$, equipped with the  multiplication  $$(i,a,\lambda)(j,b,\mu )
=\left\{\begin{array}{cr}
                                                       (i,ap_{\lambda j}b,\mu) &\mbox{if}\quad p_{\lambda j}\not=0\\
                                                       0 &\mbox{if}\quad p_{\lambda j}=0,
                                                     \end{array}\right.$$
               $$(i,a,\lambda)0=0(i,a,\lambda)=0.$$

                Fix $i_0\in I$, $\lambda_0\in\Lambda$ and define $f:I\rightarrow (0,\infty)$ by  $$f(i)=\left\{ \begin{array}{lr}

                                                     \omega(i,p_{\lambda_0i}^{-1},\lambda_0) & \mbox{if}\quad p_{\lambda_0i}\not=0\\
                                                     \omega(i,1,\lambda_0) & \mbox{if}\quad p_{\lambda_0i}=0.
                                                     \end{array}\right.
$$
Let $(i_n)$ be a sequence of distinct elements in  $I$ and set  $$x_n=\left\{ \begin{array}{lr}

                                                     (i_n,p_{\lambda_0i_n}^{-1},\lambda_0) & \mbox{if}\quad p_{\lambda_0i_n}\not=0\\
                                                     (i_n,1,\lambda_0) & \mbox{if}\quad p_{\lambda_0i_n}=0.
                                                     \end{array}\right.
$$

It is readily verified that if $p_{\lambda_0i_n}\not=0$ then $x_nx_m=x_n,$ for all $m\in\mathbb{N};$ indeed  \[x_nx_m=(i_n,p_{\lambda_0i_n}^{-1}, \lambda_0)(i_m,p_{\lambda_0i_m}^{-1}, \lambda_0 )
=(i_n,p_{\lambda_0i_n}^{-1}p_{\lambda_0i_m}p_{\lambda_0i_m}^{-1},\lambda_0)=(i_np_{\lambda_0i_n}^{-1}, \lambda_0)=x_n.\]
And if $p_{\lambda_0i_n}=0$ then $x_nx_m=0,$ for all $m\in\mathbb{N}.$

 Hence $\frac{1}{f(i_m)}=\frac{1}{\omega(x_m)}=\frac{\omega(x_nx_m)}{\omega(x_n)\omega(x_m)}=\Omega(x_n,x_m)$ in the case where $ p_{\lambda_0i_n}\not=0$ and $(\frac{\omega(0)}{f(i_m)})^2=(\frac{\omega(0)}{\omega(x_m)})^2=\frac{\omega(0)}{\omega(x_n)\omega(x_m)}$ whenever $p_{\lambda_0i_n}=0.$ These observations together with the $0-$clusterlity of $\Omega$ imply that  $(f(i_m))$ is unbounded. Hence  $I$ is countable, by Lemma \ref{countable1}. Similarly $\Lambda$ is countable. We are going to show that $G$ is also countable. To this end, let  $\omega_0(g)=\omega(i_0,gp_{\lambda_0,i}^{-1},\lambda_0)\quad (g\in G)$. Then $\omega_0$ is a weight  on $G$ such that $\Omega_0$ is $0-$cluster and so  $G$ is countable, by Theorem \ref{main}. Therefore $S$ is countable as claimed.  Proof for the case  that $S$ completely simple semigroup is similar.
\end{proof}

\noindent{{\bf Acknowledgments.}}  This research was supported by
the Centers of Excellence for Mathematics at the University of
Isfahan.

\end{document}